\documentclass[reqno,12pt]{article}
\usepackage{amsmath}
\usepackage{multirow}
\usepackage{xcolor}
\usepackage{url}
\usepackage{hyperref}
\usepackage{amsthm}
\usepackage{amssymb}
\usepackage{enumerate}
\usepackage[shortlabels]{enumitem}
\usepackage{colortbl}
\usepackage{tikz}
\usetikzlibrary{tikzmark}
\usepackage{authblk}
\makeatletter
\renewcommand{\p@enumii}{} %odnosnik do numeru subitem bez numeru nadrzednego item
\renewcommand{\p@enumiii}{}
\makeatother

\usepackage{diagbox}

\providecommand{\definitionname}{Definition}

\newtheorem{theorem}{Theorem} [section]
\newtheorem{obs}[theorem]{Observation}
\newtheorem{lemma}[theorem]{Lemma}

\newtheorem{cor}[theorem]{Corollary}
\newtheorem{conj}[theorem]{Conjecture}

\theoremstyle{definition}
\newtheorem{exe}[theorem]{Exercise}

\newcommand*{\doi}[1]{\href{http://dx.doi.org/#1}{\texttt{doi:#1}}}

\newcommand{\arxiv}[1]{\href{https://arxiv.org/abs/#1}{\texttt{arXiv:#1}}}

\usepackage{graphicx}
\def\N{\mathbb N}
\def\R{\mathbb R}
\def\RR{\tilde R}  
\def\rr{\tilde{r}}

\def\cC{\mathcal C}

\def\cH{\mathcal H}

\usepackage{float}
\begin{document}

\title{Lower bounds for online size Ramsey numbers for paths}
%\author{?}
\author{Natalia Adamska}
\author{Grzegorz Adamski}
\affil{Faculty of Mathematics and CS, Adam Mickiewicz University in Pozna\'n}

\maketitle

\begin{abstract}
Given two graphs $H_1$ and $H_2$, an online Ramsey game is  
played on the edge set of $K_\N$. In every round Builder selects an edge and Painter colors 
it red or blue.  Builder is trying to force Painter to create a red copy of $H_1$ or a blue copy of
$H_2$ as soon as possible, while Painter's goal is the opposite. The online (size) Ramsey number $\rr(H_1,H_2)$ is the smallest number of rounds in the game
provided Builder and Painter play optimally. Let $v(G)$ be the number of vertices in the graph $G$ and $v_1(G)$ be the number of vertices of degree 1 in $G$. 
We prove that if $G$ has no isolated vertices, then $\rr(P_7,G)\ge 8v(G)/5-v_1(G)$, $\rr(P_8,G)\ge 18v(G)/11-v_1(G)$ and $\rr(P_9,G)\ge 5v(G)/3-v_1(G)$. In particular $\rr(P_9,P_n)\ge 5n/3-2,$ which with known upper bound implies $\lim_{n\to\infty} \rr(P_9,P_n)/n=5/3.$ We also show that for any fixed $k$, $\lim_{n\to\infty} \rr(P_k,P_n)/n$ exists. 
\end{abstract}

\section{Introduction}

Let $\cH_1$ and $\cH_2$ be nonempty families of finite graphs. Consider the following game  $\RR(\textcolor{red}{\cH_1},\textcolor{blue}{\cH_2})$
between Builder and Painter, played on the infinite board $K_\N$ (i.e. the board is an infinite complete graph with the vertex set $\N$). In every round, Builder chooses a previously unselected
edge of $K_\N$ and Painter colors it red or blue. The game ends if after a move of Painter there is a red copy of a graph from $\cH_1$ 
or a blue copy of a graph from $\cH_2$.  Builder tries to finish the game as soon as possible, while the goal of Painter is the opposite.
Let $\rr(\textcolor{red}{\cH_1},\textcolor{blue}{\cH_2})$ be the minimum number of rounds in the game $\RR(\textcolor{red}{\cH_1},\textcolor{blue}{\cH_2})$, provided both players play optimally. 
If $\cH_i$ consists of one graph $H_i$, we simply write  $\rr(\textcolor{red}{H_1},\textcolor{blue}{H_2})$ and $\RR(\textcolor{red}{H_1},\textcolor{blue}{H_2}).$ We call them an online size Ramsey 
number and an online size Ramsey game, respectively. 

%\textcolor{magenta}{The classic game 
%$\hat{R}(\textcolor{red}{\cH_1},\textcolor{blue}{\cH_2})$
%is similar to the Builder-Painter game. However, instead of alternating turns, Builder selects 
%$n$ edges, after which Painter colors them red or blue. Builder wins if the board contains a red graph from $\cH_1$ or a blue graph from $\cH_2$. Otherwise, Painter wins. The size Ramsey number $\hat{r}(\textcolor{red}{\cH_1},\textcolor{blue}{\cH_2})$ is the smallest number
%$n$ that Builder has a winning strategy.} 
The size Ramsey number $\hat{r}(\textcolor{red}{G},\textcolor{blue}{H})$ is the smallest number $m$ such that there exists a graph $J$ with $m$ edges, where every two-coloring of the edges of $J$ (using red and blue) contains a red copy of the graph ${G}$ or a blue copy of the graph ${H}.$
It's easy to see that $\rr(\textcolor{red}{G},\textcolor{blue}{H})\le\hat{r}(\textcolor{red}{G},\textcolor{blue}{H})$.

One of the most famous open problems about the online Ramsey numbers is Rödl's conjecture (cf. \cite{ar}). It states that %the product of the online Ramsey number and the classical Ramsey number for cliques on $n$ vertices approaches zero as $n$ tends to infinity, i.e.
\[\lim_{t \rightarrow \infty} \frac{\tilde{r}(\textcolor{red}{K_t},\textcolor{blue}{K_t})}{\hat{r}(\textcolor{red}{K_t},\textcolor{blue}{K_t})} = 0,\]
where $K_t$ is a clique on $t$ vertices.

In 2009, Conlon \cite{con} showed that there are infinitely many values of $t$ for which the online
Ramsey number is exponentially smaller than the size Ramsey number,
i.e. there exists $c<1$ such that for infinitely many $t$ we have
\[\tilde{r}(\textcolor{red}{K_t},\textcolor{blue}{K_t})\le c^{t} \hat{r}(\textcolor{red}{K_t},\textcolor{blue}{K_t}).\]
In particular, he obtained
\[\liminf_{t \rightarrow \infty} \frac{\tilde{r}(\textcolor{red}{K_t},\textcolor{blue}{K_t})}{\hat{r}(\textcolor{red}{K_t},\textcolor{blue}{K_t})} = 0.\]
But it is still unknown whether the limit superior is also 0.

Let us consider a similar bound
\begin{equation}\label{1}
    \lim_{n \rightarrow \infty} \frac{\tilde{r}(\textcolor{red}{P_k},\textcolor{blue}{P_n})}{n},\tag{$*$}
\end{equation}
where instead of cliques we take paths $P_n$ on $n$ vertices.
It is worth mentioning here that $\hat{r}(\textcolor{red}{P_k},\textcolor{blue}{P_n}) =\Theta (n)$ (cf. \cite{beck2}).

Cyman, Dzido, Lapinskas and Lo~\cite{lo} proved that the above limit for $k=3$ and $k=4$ is equal to $1.25$ and $1.4$ respectively. For $k\ge 5$ they obtained that the limit inferior is at least $1.5.$ Then they conjectured that (\ref{1}) exist for any fixed $k\ge 5$ and is also equal to 1.5. Recently, Mond and Portier \cite{adva} disproved that conjecture by showing that 
$\rr(\textcolor{red}{P_{10}},\textcolor{blue}{P_n})\ge 1.(6)n-2$. This matches (up to a constant for a fixed $k$) the upper bound $\rr(\textcolor{red}{P_k},\textcolor{blue}{P_n})\le 1.(6)n+12k$ found by Bednarska-Bzd\c{e}ga~\cite{Gosia}. 

We show the limit (\ref{1}) exists for every $k\in \N$ by proving Theorem \ref{granicka}. %which was unknown for $k\in\{5,6,7,8,9\}$.  
\begin{theorem}\label{granicka}
    For any bipartite graph $H$, the limit
    $$\lim_{n\to\infty} \frac{\rr(\textcolor{red}{H},\textcolor{blue}{P_n)}}{n}$$
    exists (and is finite).
\end{theorem}

We also improve the result %obtained by Mond and Portier 
from \cite{adva} by
showing the following theorem. %IN THIS THEOREM WE USE METHOD SIMILAR TO []. DIFERENCE: SIMILIRAITY: POTENTIAL

\begin{theorem}\label{pkpcn}
Let $G$ be any graph without isolated vertices. Then
    \begin{align*}
    \rr(\textcolor{red}{P_7},\textcolor{blue}{G})&\ge 8v(G)/5-v_1(G),\\
    \rr(\textcolor{red}{P_8},\textcolor{blue}{G})&\ge 18v(G)/11-v_1(G),\\
    \rr(\textcolor{red}{P_9},\textcolor{blue}{G})&\ge 5v(G)/3-v_1(G),
    \end{align*}
    where $v(G)$ is the number of vertices in $G,$ while $v_1(G)$ is the number of vertices of degree $1$ in $G$.
\end{theorem}
%This theorem is true for any graph $G$, however, it is optimized  
%it gives better lower bounds (than known) 
%for graphs, where most vertices have degree 2. 
Its direct consequence is the following result, which (with the result of \cite{Gosia}) implies that the limit~(\ref{1}) exists and is equal to $1.(6)$ for $k\ge 9$.
\begin{cor}\label{pkpn}
For every $n\in\N$
\begin{align*}
    \rr(\textcolor{red}{P_7},\textcolor{blue}{P_n})&\ge 1.6n-2,\\
    \rr(\textcolor{red}{P_8},\textcolor{blue}{P_n})&\ge 1.(63)n-2,\\
    \rr(\textcolor{red}{P_9},\textcolor{blue}{P_n})&\ge 1.(6)n-2. 
    \end{align*}
\end{cor} 

The main tool used in the above theorem is the following lemma, which is of independent interest.
    \begin{lemma}\label{lemma}
    Let $\cC$ be the family of all cycles and $\cH(\alpha,\beta,x)$ be the family of all graphs $H$ without isolated vertices satisfying $\alpha v(H)+\beta v_1(H)\ge x$, where $v_1(H)$ is the number of vertices of degree $1$ in $H$. Then
    \begin{align}
    \rr(\textcolor{red}{\cC\cup \{P_7\}},\textcolor{blue}{\cH(8/5,-1,x)})\ge x,\\
    \rr(\textcolor{red}{\cC\cup \{P_8\}},\textcolor{blue}{\cH(18/11,-1,x)})\ge x,\\
    \rr(\textcolor{red}{\cC\cup \{P_9\}},\textcolor{blue}{\cH(5/3,-1,x)})\ge x.  
    \end{align}
\end{lemma}

To prove this lemma, we use a method similar to the one obtained by Mond and Portier \cite{adva}. We also create a potential function, but with modifications. The main idea of the proof remains the same: Painter wants the potential to increase slowly. This guarantees that the specified graphs will not appear too quickly. From this, we obtain a lower bound for the online Ramsey numbers for the games in the lemma.

Organization of the paper: The proof of the Lemma \ref{lemma} is presented in sections 3–5. This lemma is then used to prove Theorem \ref{pkpcn}, whose proof appears in section 6. Theorem \ref{granicka} is proved in section 7. The appendices include: a description of how the Painter’s strategy and the potential function were determined, a table summarizing the best known upper and lower bounds for the number $\rr(P_k,P_n)$, and a few exercises that allow the reader to apply the knowledge gained from the article in practice.

%\textbf{Organization of the paper.} In section 2

%To proof this lemma we used a method similar to published by Mond and Portier \cite{adva}. We also create a potential function, but modified. Main idea of proof remains -- Painter wants to the potential increase not to fast. It guarantees that specified graphs will not apprers to fast. From that we get estimate from below for online Ramsey numbers for games from lemma.

\section{Notation}

Let $\cC$ be the family of all cycles and $\cH(\alpha,\beta,x)$ be the family of all graphs $H$ without isolated vertices satisfying $\alpha v(H)+\beta v_1(H)\ge x$, where $v(H)$ is the number of vertices in $H$ and $v_1(H)$ is the number of vertices of degree 1 in $H$. Let %subgraph consisting of all vertices and red edges be called red subgraph. Analogously, we define blue subgraph. Let
\begin{itemize}
\item $\RR(\textcolor{red}{\cC\cup \{P_7\}},\textcolor{blue}{\cH(8/5,-1,x)})$ be called the $P_7$-game,
\item $\RR(\textcolor{red}{\cC\cup \{P_8\}},\textcolor{blue}{\cH(18/11,-1,x)})$ be called the $P_8$-game, 
\item $\RR(\textcolor{red}{\cC\cup \{P_9\}},\textcolor{blue}{\cH(5/3,-1,x)})$ be called the $P_9$-game.
\end{itemize}
Given a partially played Ramsey game, by the \emph{host graph} we mean the colored graph $G$, where $V(G)=\mathbb{N}$ and $E(G)$ is the set of all red and blue edges already colored on the board. Let graph consisting of all vertices and red edges of the host graph be called the \emph{red host graph}. Analogously, we define the \emph{blue host graph}. We divide vertices in the red host graph depending on the connected component. Each of these components can be either:
\begin{itemize}
    \item one vertex (O-shaped),
    \item $P_2$ path (I-shaped),
    \item $P_3$ path (L-shaped),
    \item $P_4$ path (N-shaped),
%    \item $K_{1,3}$ star (Y-shaped),
    \item some other graph (for example F-shaped).
\end{itemize}

Let $O,I,L,N$ be the sets of vertices in these categories, respectively. Let us also define the set $F,$ which differs slightly depending on the game we are considering. 
\begin{itemize}
    \item $F=V\setminus (O\cup I\cup L)$ in the $P_7$-game,
    \item $F=V\setminus (O\cup I\cup L\cup N)$ in the $P_8$-game,
    \item $F=V\setminus (O\cup I\cup L)$ in the $P_9$-game.
\end{itemize}
While the sets $O,I,L$ are used in all three games,  $N$ is used only in the %$P_7$-game and $N$ is used only in 
$P_8$-game. Therefore $F$ is always the set of ``other'' vertices.

Red components with vertices in $F$ will be called big. For every big red component in the red host graph we will choose a capital. The capital is either a vertex (in the $P_8$-game) or an edge (in the other two games). Let $T$ be the set of capitals. Let $F_i$ be the set of vertices with distance $i$ to the capital. The distance between a vertex $u$ and an edge $xy$ is defined by the formula $dist(u,xy)=\min(dist(u,x),dist(u,y))$.

Note that if the red component has no red cycles, then the following statements are equivalent:
\begin{itemize}
    \item We can choose a capital edge in such a way, that the distance from every vertex in this component to this capital edge is less than 3.
    \item This component has no path $P_7$. 
\end{itemize}
In particular, if the $P_7$-game is not finished, then $F=F_0\cup F_1\cup F_2$. Similarly, if the $P_8$-game is not finished, then $F=F_0\cup F_1\cup F_2$. The same with the $P_9$-game and $F=F_0\cup F_1\cup F_2\cup F_3$. Moreover, let $F_{3+}=\bigcup_{i \ge 3}F_i,$ $F_{4+}=\bigcup_{i \ge 4}F_i.$ 

Let the \textit{blue degree} $d_B(v)$ of a vertex $v$ means the degree of $v$ in the blue host graph. Finally, let us split every of the sets $O,I,L,N,F_0,F_1,F_2,F_3$ into 3 subsets depending on the blue degree. For $X\in \{O,I,L,L_i,N,N_i,F,F_i\}$ let
\begin{itemize}
    \item $V^0(G)=\{v\in G: d_B(v)=0\}$, $X^0(G)=X\cap V^0(G)$,
    \item $V^1(G)=\{v\in G: d_B(v)=1\}$, $X^1(G)=X\cap V^1(G)$,
    \item $V^{2+}(G)=\{v\in G: d_B(v)\ge 2\}$, $X^{2+}(G)=X\cap V^{2+}(G)$,
    \item $X(G)=X$,
\end{itemize}
where $G$ is a host graph.  %In total we have 21 sets (types of vertices).
%\textbf{a host graph} -- red-blue graph.\\
		
        In this work, the main role will be played by the \textit{potential function} %$f$. %, which is a function that assigns a number to each host graph. 
        %We define it by the formula 
		$f(G,T)$
        where $T$ is a set of capitals in a graph $G.$ %and $c_T:V(G)\to \R$ is a function, which gives every vertex in $G$ a number depending on type of the vertex.
        If $T$ is clear from the context, we write concisely $f(G).$

        To prove the Lemma \ref{lemma}, we will use the following observation, which does not require a proof.
		\begin{obs}\label{obs}
			Let $K$ and $H$ be graphs and $f$ be a function (a potential) defined for every possible host graph. Suppose that in the game $\RR(\textcolor{red}{K},\textcolor{blue}{H})$:
			\begin{itemize}
				\item potential of empty graph is 0,
				\item for every Builder's move, Painter can color the edge in such a~way, that the potential of the host graph does not increase by more than $x$, where $x>0$,
				\item potential of every host graph that ends the game %(i.e. that contains red $\textcolor{red}{G}$ or blue $\textcolor{blue}{H}$) 
				is at least $y$.
			\end{itemize}
			Then $\rr(\textcolor{red}{K},\textcolor{blue}{H})\ge y/x.$ %\textcolor{magenta}{?}
		\end{obs}
%In the next three subsections, we will determine the potential functions for games A, B, and C, where $y$ will be equal to 20, 1, and 2, respectively. Then, we will show that the Painter has a strategy for which
%$x$ is equal to 2, 3, and 4, respectively, for the considered games. In this way, we will obtain:

\section[P9-game]{$P_9$-game}
%Przedstawimy dowód lematu 1.4 zaczynając od gry P9. Dowód podzielimy na następujące etapy
%- zdefiniowanie funkcji potencjału
%- przedstawienie strategi malarza
%- Sprawdzenie jak szybko może zmieniać się potencjał
%- podsumowanie.
%Dowód tego lematu dla pozostałych dówch gier P7 i P8 jest analogiczny, dlatego szczegóły w następnych dwóch rozdziałach pominiemy.
We present the proof of Lemma \ref{lemma}, starting with the $P_9$-game. It consists of the following steps:
\begin{itemize}
    \item definition of a potential function,
\item presentation of Painter's strategy,
\item analysis of the rate of potential change,
\item conclusion.
\end{itemize}
The reasoning for the $P_7$-game and
the $P_8$-game, covered in the sections $4$ and $5$, is analogous, so we omit the details.

\subsection{Potential function}\label{9potential}
Now we are ready to define a potential function for a host graph $G$ with a set of capital edges $T.$ Let $c_{T,G}:V(G)\to \R$ be a function, which assigns every vertex in $G$ a number depending on the type of the vertex and is defined by the following table. For example $c_{T,G}(v)=14$ for $v\in L^1$ or $v\in F_0^1.$

\begin{table}[h]\label{figure1}
          %  \caption*{Table \thetable}
		\begin{center}
			\begin{tabular}{|r||*{7}{c|}}\hline
				\backslashbox{\textcolor{blue}{blue}}{\textcolor{red}{red}}
				&$\cellcolor{red!25}O$&\cellcolor{red!25}$I$&\cellcolor{red!25}$L,F_0$
				&$\cellcolor{red!25}F_1$&\cellcolor{red!25}$F_2$&\cellcolor{red!25}$F_3$&\cellcolor{red!25}$F_{4+}$\\\hline\hline
				\cellcolor{blue!25}0 &0&6&8&10&12&20&$\infty$\\\hline
				\cellcolor{blue!25}1 &8&13&14&15&16&20&$\infty$\\\hline
				\cellcolor{blue!25}2+ &20&20&20&20&20&20&$\infty$\\\hline
			\end{tabular}
		\end{center}
        \end{table}

Let $f(G,T)=\infty$ if $G$ has a red cycle or the set $F_{4+}$ is nonempty. Otherwise, let
$$f(G,T)=\sum_{v\in V(G)} c_{T,G}(v).$$
If $G$ is clear from the context, we write $c_T$ instead of $c_{T,G}.$
%The general idea for Painter's strategy is not to allow $f(G,T)$ to grow too fast. We will show that for every Builder's move, Painter can choose color and update the set of capital edges in such a way that 
%$$f(G',T')-f(G,T)\le 12.$$

%If $G$ is an empty graph, then $f(G,\emptyset)=0$. If $H$ is a blue subgraph of $G$ (with isolated vertices removed), then $$f(G,T)\ge 20 |V^{2+}|+8|V^1|=20 v(H)-12 v_1(H),$$ where $v_1(H)$ is the number vertex with degree 1 in $H$. It would mean that the Painter's strategy guarantees that any blue graph $H$ with no isolated vertices will appear no sooner than after $\lceil 5 v(H)/3-v_1(H)\rceil$ moves. So,
 %$$\rr(\textcolor{red}{\cC\cup \{P_9\}},\textcolor{blue}{\cH(5/3,-1,x)})\ge x.$$
\subsection{Painter's strategy}\label{9strategy}

Painter's choice of color for the edge $uv$ depends only on the types of $u$ and $v$. The presented strategy is symmetric for $u,v$. It means that whenever we say that $u\in X$ and $v\in Z$, Painter colors an edge $uv$ in a particular color, we implicitly say that Painter does the same when $u\in Z$ and $v\in X$.

In this strategy, no edge is ever removed from $T.$ From the other side, an edge is added to $T$ if and only if a new big red component appears.

In the following four cases, Painter will color the edge $uv$ blue.

\begin{enumerate}[{\color{blue} (A)}]
    \item $u\in V^{2+}\cup F_3$;
    \item $u,v\in L\cup F$;
    \item $u\in I, v\in F_1\cup F_2$;
    \item $u\in O^0, v\in F_2.$
\end{enumerate}

Note that point $\color{blue} (B)$ guarantees that no red cycle will ever appear.

In the remaining cases, listed below, she colors the edge red. In three cases the new big red component will appear, so we have to add an edge to $T$. %We will denote by $w$ one of the neighbors of $v$ in the red subgraph (if $v\not\in O$). 

\begin{enumerate}[{\color{red}(A)}]
\setcounter{enumi}{4}
\item $u,v\in O$;
\item $u\in O, v\in I$;
\item $u\in O, v\in L$ $\to$ add $vw$ to $T$, where $w\neq u$ and $vw$ is red;
\item $u\in O, v\in F_0\cup F_1$;
\item $u\in O^1, v\in F_2$;
\item $u,v\in I$ $\to$ add $uv$ to $T$;
\item $u\in I,v\in L$ $\to$ add $uv$ to $T$;
\item $u\in I,v\in F_0$.
\end{enumerate}

It is possible that the pair $u,v$ satisfies both a blue case and a red one. For example, if $u,v\in O^2$, then the pair $u,v$ satisfies 
both {\color{blue} (A)} and {\color{red} (E)}. In those cases Painter can use any color, so let us settle on blue.

The table below shows Painter's strategy for all combinations of types of vertices $u$ and $v$ that are not considered by cases {\cellcolor{blue!10}\color{blue}(A)} and {\cellcolor{blue!10}\color{blue}(B)}.

\begin{figure}[h]
\begin{center}
\begin{tabular}{|r||*{9}{c|}}\hline
\backslashbox{$u$}{$v$}
				&$\cellcolor{red!25}O$&\cellcolor{red!25}$I$&\cellcolor{red!25}$L$
				&$\cellcolor{red!25}F_0$&\cellcolor{red!25}$F_1$&\cellcolor{red!25}$F_2$\\\hline\hline
\cellcolor{red!25}$O^0$ &\cellcolor{red!10}\color{red}(E)&\cellcolor{red!10}\color{red}(F)&\cellcolor{red!10}\color{red}(G)&\cellcolor{red!10}\color{red}(H)&\cellcolor{red!10}\color{red}(H)&\cellcolor{blue!10}\color{blue}(D)\\\hline
\cellcolor{red!25}$O^1$ &\cellcolor{red!10}\color{red}(E)&\cellcolor{red!10}\color{red}(F)&\cellcolor{red!10}\color{red}(G)&\cellcolor{red!10}\color{red}(H)&\cellcolor{red!10}\color{red}(H)&\cellcolor{red!10}\color{red}(I)\\\hline
%&\color{red}E&\color{red}F&\color{red}G&\color{red}H&\color{red}H&\color{red}O\\\hline
\cellcolor{red!25}$I$ &\cellcolor{red!10}\color{red}(F)&\cellcolor{red!10}\color{red}(J)&\cellcolor{red!10}\color{red}(K)&\cellcolor{red!10}\color{red}(L)&\cellcolor{blue!10}\color{blue}(C)&\cellcolor{blue!10}\color{blue}(C)\\\hline
\end{tabular}
%\caption{This table shows all combinations of types of vertices $u$ and $v$ which are not considered by cases {\cellcolor{blue!10}\color{blue}(A)} and {\cellcolor{blue!10}\color{blue}(B)}.}
\end{center}
\end{figure}
\subsection{Change of the potential}\label{delta9}

% The choice of color for an edge $uv$ Painter depends only on type of $u$ and $v$. The presented strategy is symmetric for $u,v$. It means that whenever we say that $u\in X$ and $v\in Z$, Painter colors an edge $uv$ on a particular color, we implicitly say that Painter does the same when $u\in Z$ and $v\in X$.
Let $G$ be a host graph with a set of capital edges $T,$ $G'$ with a set of capital edges $T'$ be a graph $G$ with an added colored edge $uv.$ Let us consider the change in potential after one round, depending on the color of the edge painted by Painter. We start with assuming that Builder selected edge $uv$ and Painter colored it blue. 

% In this strategy Painter will never remove an edge from $T$ and add an edge to $T$ iff new big red component appears.
When the added edge $uv$ is blue, then the set of capital edges $T'$ for $G'$ is not changed, so $T'=T.$ Let $$\Delta c_T(v)=c_{T',G'}(v)-c_{T,G}(v).$$
%We will start by checking cases when Painter will use blue color. 
Since the added edge $uv$ is blue, $\Delta c_T(x)\neq 0$ holds only for $x\in \{u,v\}$.  If $x\in V^0(G)$ or $x\in V^1(G)$ or $x\in V^{2+}(G)$ then $x\in V^1(G'),$ $x\in V^{2+}(G'),$ $x\in V^{2+}(G')$ respectively (look at the blue arrows in the table below).    

		\begin{center}
			\begin{tabular}{|r||*{7}{c|}}\hline
				\backslashbox{\textcolor{blue}{blue}}{\textcolor{red}{red}}
				&$\cellcolor{red!25}O$&\cellcolor{red!25}$I$&\cellcolor{red!25}$L,F_0$
				&$\cellcolor{red!25}F_1$&\cellcolor{red!25}$F_2$&\cellcolor{red!25}$F_3$&\cellcolor{red!25}$F_{4+}$\\\hline\hline
				\cellcolor{blue!25}0 &0\tikzmark{1a}&6\tikzmark{2a}&8\tikzmark{3a}&10\tikzmark{4a}&12\tikzmark{5a}&20\tikzmark{6a}&$\tikzmark{7a}\infty$\\\hline
				\cellcolor{blue!25}1 &8\tikzmark{1b}&13\tikzmark{2b}&14\tikzmark{3b}&15\tikzmark{4b}&16\tikzmark{5b}&20\tikzmark{6b}&$\infty\tikzmark{7b}$\\\hline
				\cellcolor{blue!25}2+ &20\tikzmark{1c}&20\tikzmark{2c}&20\tikzmark{3c}&20\tikzmark{4c}&20\tikzmark{5c}&20\tikzmark{6c}&$\infty\tikzmark{7c}$\\\hline
			\end{tabular}
            \begin{tikzpicture}[overlay, remember picture, shorten >=.5pt, shorten <=.5pt, transform canvas={yshift=.25\baselineskip}, color=blue,line width=0.4mm]
				\draw [->] ({pic cs:1a}) [bend left] to ({pic cs:1b}) ;
				\draw [->] ({pic cs:1b}) [bend left] to ({pic cs:1c});
				\draw [->] ({pic cs:2a}) [bend left] to ({pic cs:2b});
				\draw [->] ({pic cs:2b}) [bend left] to ({pic cs:2c});
				\draw [->] ({pic cs:3a}) [bend left] to ({pic cs:3b});
				\draw [->] ({pic cs:3b}) [bend left] to ({pic cs:3c});
				\draw [->] ({pic cs:4a}) [bend left] to ({pic cs:4b});
				\draw [->] ({pic cs:4b}) [bend left] to ({pic cs:4c});
				\draw [->] ({pic cs:5a}) [bend left] to ({pic cs:5b});
				\draw [->] ({pic cs:5b}) [bend left] to ({pic cs:5c});
				\draw [->] ({pic cs:6a}) [bend left] to ({pic cs:6b});
				\draw [->] ({pic cs:6b}) [bend left] to ({pic cs:6c});                
				\end{tikzpicture}
		\end{center}
Now we check case~\textcolor{blue}{(D)}. This means that $u\in O^0(G),$ so $u\in O^1(G')$ and 
$\max\Delta c_T(v)=8-0=8.$ We also know that $u\in F_2,$ so
$$\max\Delta c_T(u)=\max\{\Delta c_T(u): u\in F_2(G)\}=\max\{16-12,20-16,20-20\}=4.$$ From this we obtain that the difference of potential $\Delta f$ (i.e. $f(G',T')-f(G,T)$) is at most 12 in case~\textcolor{blue}{(D)} (look at the blue table).

Analogously, by checking cases \textcolor{blue}{(A)}, \textcolor{blue}{(B)} and \textcolor{blue}{(C)}, we get $\Delta f\le 12$ for every Painter's ``blue move".
\begin{center}

\begin{tabular}{|*{6}{c|}}\hline

\cellcolor{blue!25}case&\cellcolor{blue!25}$u$&\cellcolor{blue!25}$v$&\cellcolor{blue!25} $\max \Delta c_T(u)$ &\cellcolor{blue!25}$\max \Delta c_T(v)$ &\cellcolor{blue!25} $\max\Delta f$ \\\hline\hline
\rowcolor{blue!5}{\color{blue}(A)} &$V^{2+}\cup F_3$&any&0&12&12\\\hline
\rowcolor{blue!5}{\color{blue}(B)} &$L\cup F$&$L\cup F$&6&6&12\\\hline
\rowcolor{blue!5}{\color{blue}(C)} &$I$&$F_1\cup F_2$&7&5&12\\\hline
\rowcolor{blue!5}{\color{blue}(D)} &$O^0$&$F_2$&8&4&12\\\hline
\end{tabular}
\end{center}
Now we check cases, when Painter colors $uv$ red. 
This situation is a bit more complicated -- $c_T(x)$ may change not only for $u,v$, but for all vertices in the red components of $u$ and $v$. However, vertices in big red components are not affected by the new red edge. In the following red table, we use a notation, showing ``worst-case scenario'' of each case, i.e. the subcase where the increase of function $f$ is the biggest. For example, $O,2I\to 3L$ means that one vertex from $O(G)$ and two vertices from $I(G)$ will move to the set $L(G')$ in the worst-case scenario.

It is also worth noting one property of the function $c_T$. It is not necessary in the proof, but it helps in reducing the number of cases. If a vertex $u$ moves from $X^i(G)$ to $Z^i(G')$ and a vertex $v$ moves from $X^j(G)$ to $Z^j(G')$ after adding a red edge, then $\Delta c_T(u)\ge \Delta c_T(v)$, for $i<j$. It means that the worst-case scenario for Painter (i.e. $\Delta f$ is the biggest) is when the blue degree of all vertices in the red components of $u$ and $v$ are the smallest possible. In cases $\color{red}(E)$--$\color{red}(L)$ we may assume that all those vertices are not incident to blue edges besides case \textcolor{red}{(I)}, where $v\in O^1(G)$ have one blue edge.

\begin{center}
\begin{tabular}{|c|c|c|c|c|p{59.0mm}|c|}\hline

\cellcolor{red!25}case&\cellcolor{red!25}$u$&\cellcolor{red!25}$v$&\cellcolor{red!25} \parbox[c]{2cm}{edge \\added to $T$} &\cellcolor{red!25}worst-case scenario & \multicolumn{2}{c|}{\cellcolor{red!25}$\max\Delta f$} \\\hline\hline
\rowcolor{red!5}{\color{red}(E)} &$O$&$O$&&$2O\to 2I$&$-(2\cdot 0)+(2\cdot 6)$&12\\\hline
\rowcolor{red!5}{\color{red}(F)} &$O$&$I$&&$O,2I\to 3L$&$-(0+2\cdot 6)+(3\cdot 8$)&12\\\hline
\rowcolor{red!5}{\color{red}(G)} &$O$&$L$&$vw$&$O,3L\to 2F_0,2F_1$&$-(0+3\cdot 8)+(2\cdot 8+2\cdot 10)$&12\\\hline
\rowcolor{red!5}{\color{red}(H)} &$O$&$F_0\cup F_1$&&$O\to F_2$&$-(0)+(12)$&12\\\hline
\rowcolor{red!5}{\color{red}(I)} &$O^1$&$F_2$&&$O^1\to F_3^1$&$-(8)+(12)$&12\\\hline
\rowcolor{red!5}{\color{red}(J)} &$I$&$I$&$uv$&$4I\to 2F_0,2F_1 $&$-(4\cdot 6)+(2\cdot 8+2\cdot 10)$&12\\\hline
\rowcolor{red!5}{\color{red}(K)} &$I$&$L$&$uv$&$2I,3L\to 2F_0,2F_1,F_2$&$-(2\cdot 6+3\cdot 8)+(2\cdot 8+2\cdot 10+12)$&12\\\hline
\rowcolor{red!5}{\color{red}(L)} &$I$&$F_0$&&$2I\to F_1,F_2$&$-(2\cdot 6)+(10+12)$&10\\\hline
\end{tabular}
\end{center}

Looking at the table, we see that the potential does not increase by more than 12 when the edge $uv$ is red.

Summarizing, Painter can color the edge $uv$ such that 
$$\Delta f=f(G',T')-f(G,T)\le 12.$$
\subsection{Proof of Lemma \ref{lemma} (3)}
%In section \ref{9strategy} we presented Painter's strategy and showed that it is well-defined. Futhermore, in section \ref{9potential} we obtain that if the potential does not change by more than 12 after adding one edge (as we have already proved in section \ref{delta9}), then 
%$$\rr(\textcolor{red}{\cC\cup \{P_9\}},\textcolor{blue}{\cH(5/3,-1,x)})\ge x.$$
In section \ref{9strategy} we presented Painter's strategy and showed that it is well-defined. From the definition of the potential, we know that $f(G,T)=\infty$ if $G$ contains a red cycle or $F_{4+},$ and thus, in particular, if it contains a red path $P_9.$ In section \ref{delta9} we obtained that the potential does not change by more than 12 after one round, thus a red cycle and a red path $P_9$ will never appear in this game. To prove that
$$\rr(\textcolor{red}{\cC\cup \{P_9\}},\textcolor{blue}{\cH(5/3,-1,x)})\ge x,$$
it remains to show that no graph from $\cH(5/3,-1,x)$ will appear sooner than after $x$ moves.

Let $H\in \cH(5/3,-1,x).$ It means that $5 v(H)/3- v_1(H)\ge x.$
For every graph $K,$ which contains a blue $H$ we have (from section \ref{9potential}) $$f(K,T)\ge 20 |V^{2+}(K)|+8|V^1(K)|=20 v(H)-12 v_1(H)\ge 12x.$$
Of course $f(\emptyset)=0.$ Furthermore, the potential does not change by more than 12 after one round. Using Observation \ref{obs} we conclude that $H$ will appear no sooner than after $x$ moves.

%If $G$ is an empty graph, then $f(G,\emptyset)=0$. Let $H$ be a blue subgraph of $G$ (with isolated vertices removed). From Table \ref{figure1} we have $$f(G,T)\ge 20 |V^{2+}|+8|V^1|=20 v(H)-12 v_1(H),$$ where $v_1(H)$ is the number vertex with degree 1 in $H$. Furthermore, in section \ref{delta9} we obtain that the potential does not change by more than 12 after adding one edge.
%It would mean that the Painter's strategy guarantees that any blue graph $H$ will appear no sooner than after $\lceil 5 v(H)/3-v_1(H)\rceil$ moves (Observation \ref{obs}). So,

\section[P7-game]{$P_7$-game}
The proof of the Lemma \ref{lemma} for $P_7$-game is very similar to that for $P_9$-game. Therefore, we will omit most of the details. Based on the idea from the previous section and using the tables below for $P_7$-game, one can quickly reconstruct the proof.
\subsection{Potential function}\label{7potential}
Let $T$ be a set of capital edges in $G$ and $c_T:V(G)\to \R$ be defined by the following table, where $L_0$ is a set of central vertices in red $P_3$ paths and $L_1=L\setminus L_0$.%$Y$ is the set of vertices of $K_{1,3}$ red stars and $F$ contains vertices all vertices which connected components have at least 4 vertices except of ones in $Y$.
\begin{center}
\begin{tabular}{|r||*{7}{c|}}\hline
\backslashbox{\textcolor{blue}{blue}}{\textcolor{red}{red}}
&\cellcolor{red!25}$O$&\cellcolor{red!25}$I$&\cellcolor{red!25}$L_0$&\cellcolor{red!25}{$L_1,F_0$}
&\cellcolor{red!25}$F_1$&\cellcolor{red!25}$F_2$&\cellcolor{red!25}$F_{3+}$\\\hline\hline
\cellcolor{blue!25}0 &0&10&12&14&16&32&$\infty$\\\hline
\cellcolor{blue!25}1 &12&21&22&23&24&32&$\infty$\\\hline
\cellcolor{blue!25}2+ &32&32&32&32&32&32&$\infty$\\\hline
\end{tabular}
\end{center}

Let $f(G,T)=\infty$ if $G$ has a red cycle or the set $F_{3+}$ is nonempty. Otherwise, let
$$f(G,T)=\sum_{v\in V(G)} c_T(v).$$

\subsection{Painter's strategy}\label{7strategy}

In the following four cases Painter will color the edge $uv$ blue.

\begin{enumerate}[{\color{blue} (A)}]
    \item $u\in V^{2+}\cup F_2$;
    \item $u,v\in L\cup F$;
    \item $u\in I, v\in L_1\cup F$;
    \item $u\in O^0, v\in F_1.$
\end{enumerate}

Note that point $\color{blue} (B)$ guarantees that no red cycle will ever appear.

In the remaining cases she colors the edge red. In three cases a new big red component will appear, so we have to add an edge to $T$. %We will denote by $w$ one of the neighbors of $v$ in the red subgraph (if $v\not\in O$). 

\begin{enumerate}[{\color{red}(A)}]
\setcounter{enumi}{4}
\item $u,v\in O$;
\item $u\in O, v\in I$;
\item $u\in O, v\in L$, $\to$ add $vw$ to $T$, where $w\neq u$ and $vw$ is red;
\item $u\in O, v\in F_0$;
\item $u\in O^1, v\in F_1$;
\item $u,v\in I$ $\to$ add $uv$ to $T$;
\item $u\in I,v\in L_0$ $\to$ add $uv$ to $T$.
\end{enumerate}

\begin{figure}[h]
\begin{center}
\begin{tabular}{|r||*{9}{c|}}\hline
\backslashbox{$u$}{$v$}
&\cellcolor{red!25}$O$&\cellcolor{red!25}$I$&\cellcolor{red!25}$L_0$&\cellcolor{red!25}{$L_1$}
&\cellcolor{red!25}$F_0$&\cellcolor{red!25}$F_1$\\\hline\hline
\cellcolor{red!25}$O^0$ &\cellcolor{red!10}\color{red}(E)&\cellcolor{red!10}\color{red}(F)&\cellcolor{red!10}\color{red}(G)&\cellcolor{red!10}\color{red}(G)&\cellcolor{red!10}\color{red}(H)&\cellcolor{blue!10}\color{blue}(D)\\\hline
\cellcolor{red!25}$O^1$ &\cellcolor{red!10}\color{red}(E)&\cellcolor{red!10}\color{red}(F)&\cellcolor{red!10}\color{red}(G)&\cellcolor{red!10}\color{red}(G)&\cellcolor{red!10}\color{red}(H)&\cellcolor{red!10}\color{red}(I)\\\hline
\cellcolor{red!25}$I$ &\cellcolor{red!10}\color{red}(F)&\cellcolor{red!10}\color{red}(J)&\cellcolor{red!10}\color{red}(K)&\cellcolor{blue!10}\color{blue}(C)&\cellcolor{blue!10}\color{blue}(C)&\cellcolor{blue!10}\color{blue}(C)\\\hline
\end{tabular}
\caption{This table shows all combinations of types of vertices $u$ and $v$ which are not considered by cases {\color{blue}(A)} and {\color{blue}(B)}.}
\end{center}
\end{figure}

\subsection{Change of potential}\label{delta7}
Let $c_T:V(G)\to \R$ be a function defined by the following table.
\begin{center}
\begin{tabular}{|r||*{7}{c|}}\hline
\backslashbox{\textcolor{blue}{blue}}{\textcolor{red}{red}}
&\cellcolor{red!25}$O$&\cellcolor{red!25}$I$&\cellcolor{red!25}$L_0$&\cellcolor{red!25}{$L_1,F_0$}
&\cellcolor{red!25}$F_1$&\cellcolor{red!25}$F_2$&\cellcolor{red!25}$F_{3+}$\\\hline\hline
\cellcolor{blue!25}0 &0&10&12&14&16&32&$\infty$\\\hline
\cellcolor{blue!25}1 &12&21&22&23&24&32&$\infty$\\\hline
\cellcolor{blue!25}2+ &32&32&32&32&32&32&$\infty$\\\hline
\end{tabular}
\end{center}

We carry out the proof analogously to the reasoning from section \ref{delta9}, using the two tables below.
    \begin{center}
\begin{tabular}{|*{6}{c|}}\hline

\cellcolor{blue!25}case&\cellcolor{blue!25}$u$&\cellcolor{blue!25}$v$&\cellcolor{blue!25} $\max \Delta c_T(u)$ &\cellcolor{blue!25}$\max \Delta c_T(v)$ &\cellcolor{blue!25} $\max\Delta f$ \\\hline\hline
\rowcolor{blue!5}{\color{blue}(A)} &$V^{2+}\cup F_2$&any&0&20&20\\\hline
\rowcolor{blue!5}{\color{blue}(B)} &$L\cup F$&$L\cup F$&10&10&20\\\hline
\rowcolor{blue!5}{\color{blue}(C)} &$I$&$ L_1\cup F$&11&9&20\\\hline
\rowcolor{blue!5}{\color{blue}(D)} &$O^0$&$F_1$&12&8&20\\\hline
\end{tabular}
\end{center}

\begin{center}
\begin{tabular}{|*{7}{c|}}\hline

\cellcolor{red!25}case&\cellcolor{red!25}$u$&\cellcolor{red!25}$v$&\cellcolor{red!25} \parbox[c]{2cm}{edge \\added to $T$} &\cellcolor{red!25}worst-case scenario & \multicolumn{2}{c|}{\cellcolor{red!25}$\max\Delta f$} \\\hline\hline
\rowcolor{red!5}{\color{red}(E)} &$O$&$O$&&$2O\to 2I$&$-(2\cdot 0)+(2\cdot 10)$&20\\\hline
\rowcolor{red!5}{\color{red}(F)} &$O$&$I$&&$O,2I\to L_0,2L_1$&$-(0+2\cdot 10)+(12+2\cdot 14)$&20\\\hline
\rowcolor{red!5}{\color{red}(G)} &$O$&$L$&$vw$&$O,L_0,2L_1\to 2F_0,2F_1$&$-(0+12+2\cdot 14)+(2\cdot 14+2\cdot 16)$&20\\\hline
\rowcolor{red!5}{\color{red}(H)} &$O$&$F_0$&&$O\to F_1$&$-(0)+(16)$&16\\\hline
\rowcolor{red!5}{\color{red}(I)} &$O^1$&$F_1$&&$O^1\to F_2^1$&$-(12)+(32)$&20\\\hline
\rowcolor{red!5}{\color{red}(J)} &$I$&$I$&$uv$&$4I\to 2F_0,2F_1 $&$-(4\cdot 10)+(2\cdot 14+2\cdot 16)$&20\\\hline
\rowcolor{red!5}{\color{red}(K)} &$I$&$L_0$&$uv$&$2I,L_0,2L_1\to 2F_0,3F_1$&$-(2\cdot 10+12+2\cdot 14)+(2\cdot 14+3\cdot 16)$&16\\\hline

\end{tabular}
\end{center}
Summarizing, Painter can color the edge $uv$ such that 
$$\Delta f=f(G',T')-f(G,T)\le 20.$$

\subsection{Proof of Lemma \ref{lemma} (1)}
In section \ref{7strategy} we presented Painter's strategy and showed that it is well-defined. From the definition of the potential, we know that $f(G,T)=\infty$ if $G$ contains a red cycle or $F_{3+},$ and thus, in particular, it contains a red path $P_7.$ In section \ref{delta7} we obtained that the potential does not change by more than 20 after one round, thus a red cycle and a red path $P_7$ will never appear in this game. To prove that
$$\rr(\textcolor{red}{\cC\cup \{P_7\}},\textcolor{blue}{\cH(8/5,-1,x)})\ge x.$$
it remains to show that no graph from $\cH(8/3,-1,x)$ will appear sooner than after $x$ moves.

Let $H\in \cH(8/3,-1,x).$ It means that $8 v(H)/3- v_1(H)\ge x.$
For every graph $K,$ which contains a blue $H$ we have (from section \ref{7potential}) $$f(K,T)\ge 32 |V^{2+}(K)|+12|V^1(K)|=32 v(H)-20 v_1(H)\ge 20x.$$
Of course $f(\emptyset)=0.$ Furthermore, the potential does not change by more than 20 after one round. Using Observation \ref{obs} we conclude that $H$ will appear no sooner than after $x$ moves.

\section[P8-game]{$P_8$-game}
The idea of the proof of the Lemma \ref{lemma} for $P_8$-game is similar to that for $P_9$-game. The main difference is that instead of capital edges, we have capital vertices.

\subsection{Potential function}\label{8potential}
Let $T$ be a set of capital vertices in $G$ and $c_T:V(G)\to \R$ be defined by the following table.
\begin{center}
\begin{tabular}{|r||*{11}{c|}}\hline
\backslashbox{\textcolor{blue}{blue}}{\textcolor{red}{red}}
&\cellcolor{red!25}$O$&\cellcolor{red!25}$I$&\cellcolor{red!25}$L_0$&\cellcolor{red!25}$L_1$
&\cellcolor{red!25}$N_0$&\cellcolor{red!25}$N_1$&\cellcolor{red!25}$F_0$&\cellcolor{red!25}$F_1$&\cellcolor{red!25}$F_2$&\cellcolor{red!25}$F_3$&\cellcolor{red!25}$F_{4+}$\\\hline\hline
\cellcolor{blue!25}0 & 0 & 22 & 28 & 30 & 32 & 34 & 28 & 34 & 40 & 72&$\infty$ \\\hline
\cellcolor{blue!25}1 & 28 & 47 & 50 & 52 & 54 & 53 & 50 & 53 & 56 & 72&$\infty$ \\\hline
\cellcolor{blue!25}2+ & 72 & 72 & 72 & 74 & 76 & 72 & 72 & 72 & 72 & 72&$\infty$ \\\hline
\end{tabular}
\end{center}
We can simplify the representation of this table.
\begin{center}
\begin{tabular}{|r||*{11}{c|}}\hline
\backslashbox{\textcolor{blue}{blue}}{\textcolor{red}{red}}
&\cellcolor{red!25}$O$&\cellcolor{red!25}$I$&\cellcolor{red!25}$L_0,F_0$&\cellcolor{red!25}$L_1$
&\cellcolor{red!25}$N_0$&\cellcolor{red!25}$N_1,F_1$&\cellcolor{red!25}$F_2$&\cellcolor{red!25}$F_3$&\cellcolor{red!25}$F_{4+}$\\\hline\hline
\cellcolor{blue!25}0 & 0 & 22 & 28 & 30 & 32 & 34 & 40 & 72&$\infty$ \\\hline
\cellcolor{blue!25}1 & 28 & 47 & 50 & 52 & 54 & 53 & 56 & 72&$\infty$ \\\hline
\cellcolor{blue!25}2+ & 72 & 72 & 72 & 74 & 76 & 72 & 72 & 72&$\infty$ \\\hline
\end{tabular}
\end{center}

Let $f(G,T)=\infty$ if $G$ has a red cycle or the set $F_{4+}$ is nonempty. Otherwise, let
$$f(G,T)=\sum_{v\in V(G)} c_T(v).$$

%$$(\textcolor{red}{P_8},\textcolor{blue}{P_n})\ge 18n/11-2= 1.(63)n-2$$

\subsection{Painter's strategy}\label{8strategy}

In the following four cases Painter will color the edge $uv$ blue.

\begin{enumerate}[{\color{blue} (A)}]
    \item $u\in V^{2+}\cup F_2$;
    \item $u,v\in L\cup N\cup F$;
    \item $u\in I, v\in N_1\cup F_1\cup F_2\cup F_3$;
    \item $u\in O^0, v\in F_2$.
\end{enumerate}

Note that point $\color{blue} (B)$ guarantees that no red cycle will ever appear.

In the remaining cases she colors the edge red. In five cases a new big red component will appear, so we have to add a vertex to $T$. %We will denote by $w$ one of the neighbors of $v$ in the red subgraph (if $v\not\in O$). 

\begin{enumerate}[{\color{red}(A)}]
\setcounter{enumi}{4}
\item $u,v\in O$;
\item $u\in O, v\in I$;
\item $u\in O, v\in L_0$, $\to$ add $v$ to $T$;
\item $u\in O, v\in L_1$;
\item $u\in O, v\in N_0$, $\to$ add $v$ to $T$;
\item $u\in O, v\in N_1$, $\to$ add $w$ to $T$, where $w\neq u$ and $vw$ is red;
\item $u\in O, v\in F_0\cup F_1$;
\item $u\in O^1, v\in F_2$;
\item $u,v\in I$;
\item $u\in I,v\in L$ $\to$ add $v$ to $T$;
\item $u\in I,v\in N_0$ $\to$ add $v$ to $T$;
\item $u\in I,v\in F_0$.
\end{enumerate}

\begin{figure}[h]
\begin{center}
\begin{tabular}{|r||*{9}{c|}}\hline
\backslashbox{$u$}{$v$}
&\cellcolor{red!25}$O$&\cellcolor{red!25}$I$&\cellcolor{red!25}$L_0$&\cellcolor{red!25}$L_1$
&\cellcolor{red!25}$N_0$&\cellcolor{red!25}$N_1$&\cellcolor{red!25}$F_0$&\cellcolor{red!25}$F_1$&\cellcolor{red!25}$F_2$\\\hline\hline
\cellcolor{red!25}$O^0$ &\cellcolor{red!10}\color{red}(E)&\cellcolor{red!10}\color{red}(F)&\cellcolor{red!10}\color{red}(G)&\cellcolor{red!10}\color{red}(H)&\cellcolor{red!10}\color{red}(I)&\cellcolor{red!10}\color{red}(J) & \cellcolor{red!10}\color{red}(K) &\cellcolor{red!10}\color{red}(K) & \cellcolor{blue!10}\color{blue}(D)\\\hline

\cellcolor{red!25}$O^1$ &\cellcolor{red!10}\color{red}(E)&\cellcolor{red!10}\color{red}(F)&\cellcolor{red!10}\color{red}(G)&\cellcolor{red!10}\color{red}(H)&\cellcolor{red!10}\color{red}(I)&\cellcolor{red!10}\color{red}(J) & \cellcolor{red!10}\color{red}(K) &\cellcolor{red!10}\color{red}(K) & \cellcolor{red!10}\color{red}(L)\\\hline

\cellcolor{red!25}$I$ &\cellcolor{red!10}\color{red}(F)&\cellcolor{red!10}\color{red}(M)&\cellcolor{red!10}\color{red}(N)&\cellcolor{red!10}\color{red}(N)&\cellcolor{red!10}\color{red}(O)&\cellcolor{blue!10}\color{blue}(C)&\cellcolor{red!10}\color{red}(P)&\cellcolor{blue!10}\color{blue}(C)&\cellcolor{blue!10}\color{blue}(C)\\\hline
\end{tabular}
\caption{This table shows all combinations of types of vertices $u$ and $v$ which are not considered by cases {\color{blue}(A)} and {\color{blue}(B)}.}
\end{center}
\end{figure}
\subsection{Change of potential}\label{delta8}
Let $c_T:V(G)\to \R$ be a function defined by the following table.
\begin{center}
\begin{tabular}{|r||*{11}{c|}}\hline
\backslashbox{\textcolor{blue}{blue}}{\textcolor{red}{red}}
&\cellcolor{red!25}$O$&\cellcolor{red!25}$I$&\cellcolor{red!25}$L_0,F_0$&\cellcolor{red!25}$L_1$
&\cellcolor{red!25}$N_0$&\cellcolor{red!25}$N_1,F_1$&\cellcolor{red!25}$F_2$&\cellcolor{red!25}$F_3$&\cellcolor{red!25}$F_{4+}$\\\hline\hline
\cellcolor{blue!25}0 & 0 & 22 & 28 & 30 & 32 & 34 & 40 & 72&$\infty$ \\\hline
\cellcolor{blue!25}1 & 28 & 47 & 50 & 52 & 54 & 53 & 56 & 72&$\infty$ \\\hline
\cellcolor{blue!25}2+ & 72 & 72 & 72 & 74 & 76 & 72 & 72 & 72&$\infty$ \\\hline
\end{tabular}
\end{center}
We carry out the proof analogously to the reasoning from section \ref{delta9}, using the two tables below.
    \begin{center}
\begin{tabular}{|*{6}{c|}}\hline

\cellcolor{blue!25}case&\cellcolor{blue!25}$u$&\cellcolor{blue!25}$v$&\cellcolor{blue!25} $\max \Delta c_T(u)$ &\cellcolor{blue!25}$\max \Delta c_T(v)$ &\cellcolor{blue!25} $\max\Delta f$ \\\hline\hline
\rowcolor{blue!5}{\color{blue}(A)} &$V^{2+}\cup F_2$&any&0&44&44\\\hline
\rowcolor{blue!5}{\color{blue}(B)} &$L\cup N\cup F$&$L\cup N\cup F$&22&22&44\\\hline
\rowcolor{blue!5}{\color{blue}(C)} &$I$&$ N_1\cup F_1\cup F_2\cup F_3$&25&19&44\\\hline
\rowcolor{blue!5}{\color{blue}(D)} &$O^0$&$F_2$&28&16&44\\\hline
\end{tabular}
\end{center}

\begin{center}
\begin{tabular}{|c|c|c|c|c|p{40.3mm}|c|}\hline

\cellcolor{red!25}case&\cellcolor{red!25}$u$&\cellcolor{red!25}$v$&\cellcolor{red!25} \parbox[c]{2cm}{vertex \\added to $T$} &\cellcolor{red!25}worst-case scenario & \multicolumn{2}{c|}{\cellcolor{red!25}$\max\Delta f$} \\\hline\hline
\rowcolor{red!5}{\color{red}(E)} &$O$&$O$&&$2O\to 2I$&$-(2\cdot 0)+(2\cdot 22)$&44\\\hline
\rowcolor{red!5}{\color{red}(F)} &$O$&$I$&&$O,2I\to L_0,2L_1$&$-(0+2\cdot 22)+(28+2\cdot 30)$&44\\\hline
\rowcolor{red!5}{\color{red}(G)} &$O$&$L_0$&$v$&$O,L_0,2L_1\to F_0,3F_1$&$-(0+28+2\cdot 30)+\textcolor{white}{0000000}+(28+3\cdot 34)$&42\\\hline
\rowcolor{red!5}{\color{red}(H)} &$O$&$L_1$&&$O,L_0,2L_1\to 2N_0,2N_1$&$-(0+28+2\cdot 30)+\textcolor{white}{0000}+(2\cdot 32+2\cdot 34)
$&44\\\hline

\rowcolor{red!5}{\color{red}(I)} &$O$&$N_0$&$v$&$O,2N_0,2N_1\to F_0,3F_1,F_2$&$-(0+2\cdot 32+2\cdot 34)+\textcolor{white}{000}+(28+3\cdot 34+40)$&38\\\hline
\rowcolor{red!5}{\color{red}(J)} &$O$&$N_1$&$w$&$O,2N_0,2N_1\to F_0,2F_1,2F_2$&$-(0+2\cdot 32+2\cdot 34)+\textcolor{white}{0}+(28+2\cdot 34+2\cdot 40)$&44\\\hline

\rowcolor{red!5}{\color{red}(K)} &$O$&$F_0\cup F_1$&&$O\to F_2$&$-(0)+(40)$&40\\\hline
\rowcolor{red!5}{\color{red}(L)} &$O^1$&$F_2$&&$O^1\to F_3^1$&$-(28)+(72)$&44\\\hline
\rowcolor{red!5}{\color{red}(M)} &$I$&$I$&&$4I\to 2N_0,2N_1 $&$-(4\cdot 22)+(2\cdot 32+2\cdot 34)$&44\\\hline
\rowcolor{red!5}{\color{red}(N)} &$I$&$L$&$v$&$2I,L_0,2L_1\to F_0,2F_1,2F_2$&$-(2\cdot 22 +28+2\cdot 30)+\textcolor{white}{0}+(28+2\cdot 34+2\cdot 40)$&44\\\hline

\rowcolor{red!5}{\color{red}(O)} &$I$&$N_0$&$v$&$2I,2N_0,2N_1\to F_0,3F_1,2F_2$&$-(2\cdot 22+2\cdot 32+2\cdot 34)+\textcolor{white}{0}+(28+3\cdot 34+2\cdot 40)$&34\\\hline
\rowcolor{red!5}{\color{red}(P)} &$I$&$F_0$&&$2I\to F_1,F_2$&$-(2\cdot 22)+(34+40)$&30\\\hline

\end{tabular}
\end{center}
Summarizing, Painter can color the edge $uv$ such that 
$$\Delta f=f(G',T')-f(G,T)\le 44.$$
\subsection{Proof of Lemma \ref{lemma} (2)}
In section \ref{8strategy} we presented Painter's strategy and showed that it is well-defined. From the definition of the potential, we know that $f(G,T)=\infty$ if $G$ contains a red cycle or $F_{4+},$ and thus, in particular, it contains a red path $P_8.$ In section \ref{delta8} we obtained that the potential does not change by more than 44 after one round, thus a red cycle and a red path $P_8$ will never appear in this game. To prove that
$$\rr(\textcolor{red}{\cC\cup \{P_8\}},\textcolor{blue}{\cH(18/11,-1,x)})\ge x.$$
it remains to show that no graph from $\cH(18/11,-1,x)$ will appear sooner than after $x$ moves.

Let $H\in \cH(18/11,-1,x).$ It means that $18 v(H)/11- v_1(H)\ge x.$
For every graph $K,$ which contains a blue $H$ we have (from section \ref{8potential}) $$f(K,T)\ge 72 |V^{2+}(K)|+28|V^1(K)|=72 v(H)-44 v_1(H)\ge 44x.$$
Of course $f(\emptyset)=0.$ Furthermore, the potential does not change by more than 44 after one round. Using Observation \ref{obs} we conclude that $H$ will appear no sooner than after $x$ moves.

\section{Proof of Theorem \ref{pkpcn}}

%A direct consequence of the preceding three sections is the lemma below. \begin{lemma}\label{lemma}
 %   Let $\cC$ be the set of all (simple) cycles and $\cH(\alpha,\beta,x)$ be the set of all graphs $H$ without isolated vertices satisfying $\alpha v(H)+\beta v_1(H)\ge x$, where $v_1(H)$ is the number of vertices of degree $1$ in $H$. Then
  %  \begin{align*}
   % \rr(\textcolor{red}{\cC\cup \{P_7\}},\textcolor{blue}{\cH(8/5,-1,x)})\ge x,\\
    %\rr(\textcolor{red}{\cC\cup \{P_8\}},\textcolor{blue}{\cH(18/11,-1,x)})\ge x,\\
    %\rr(\textcolor{red}{\cC\cup \{P_9\}},\textcolor{blue}{\cH(5/3,-1,x)})\ge x.  
    %\end{align*}
%\end{lemma}
%Using this lemma, we obtain the following theorem.
%\begin{theorem}\label{pkpcn}
%Let $G$ be any graph without isolated vertices. Then
%    \begin{align*}
 %   \rr(\textcolor{red}{P_7},\textcolor{blue}{G})&\ge 8v(G)/5-v_1(G),\\
  %  \rr(\textcolor{red}{P_8},\textcolor{blue}{G})&\ge 18v(G)/11-v_1(G),\\
   % \rr(\textcolor{red}{P_9},\textcolor{blue}{G})&\ge 5v(G)/3-v_1(G), 
    %\end{align*}
    %where $v_1(G)$ is the number of vertices of degree $1$ in $G$.
%\end{theorem}
%\begin{proof}
Let us take a graph $G$ without isolated vertices and $x= 8v(G)/5-v_1(G).$ From lemma~\ref{lemma} we have
$$\rr(\textcolor{red}{\cC\cup \{P_7\}},\textcolor{blue}{\cH(8/5,-1, 8v(G)/5-v_1(G))})\ge  8v(G)/5-v_1(G).$$
Of course $G\in \cH(8/5,-1, 8v(G)/5-v_1(G)),$ so
$$\rr(\textcolor{red}{P_7},\textcolor{blue}{G})\ge 8v(G)/5-v_1(G).$$
Analogously we get the remaining two inequalities of the Theorem \ref{pkpcn}.
%\end{proof}

\section {Existence of the limit of $\rr(H,P_n)/n$ for bipartite $H$}
In the game $\RR(\textcolor{red}{\cH_1},\textcolor{blue}{\cH_2})$ we assume that the host graph at the beginning of the game is empty. However, sometimes it is useful to consider a game when it is not empty. For a starting host graph $G$ let us denote this modified game as $\RR_G(\textcolor{red}{\cH_1},\textcolor{blue}{\cH_2})$ and let $\rr_G(\textcolor{red}{\cH_1},\textcolor{blue}{\cH_2})$ be the number of rounds of the game when players play optimally.

\begin{lemma}
    Let $G$ be a blue graph consisting of two disjoint paths $P_n$ and $P_m$ and $H$ be a bipartite graph. Then $\rr_G(\textcolor{red}{H},\textcolor{blue}{P_{n+m-v(H)}})\le e(H)$.
\end{lemma}
\begin{proof}
We can assume that $n>v(H)$ and $m>v(H)$, otherwise $G$ already contains a blue path $P_{n+m-v(H)}$. %be the vertices of the $P_m$ path. 
Let $a,b$ be such numbers that $H\subset K_{a,b}$ and $a+b=v(H)$. Builder takes the first $a$ vertices from the blue path $P_n$ and the first $b$ vertices from the blue path $P_m$ and in a sequence of rounds chooses edges of a copy of $H$ such that all chosen edges are between the paths $P_n$ and $P_m$. Obviously, if Painter colors all the edges red, the red $H$ will appear. Otherwise, there will be a new blue edge. It will close the path between the unused ends of the original paths $P_n$ and $P_m$. The new path contains all the vertices from the starting paths except of at most $a-1$ vertices from $P_n$ and $b-1$ vertices from $P_m$, so the new blue path have at least $n+m-v(H)+2$ vertices.
\end{proof}

\begin{cor}
Let $H$ be a bipartite graph and $n>v(H)$. Then
    $$\rr(\textcolor{red}{H},\textcolor{blue}{P_m})\le \left\lceil\frac{m}{n-v(H)}\right\rceil(\rr(\textcolor{red}{H},\textcolor{blue}{P_n})+e(H)).$$
\end{cor}
\begin{proof}
    We present a strategy for Builder. He starts the game $\RR(\textcolor{red}{H},\textcolor{blue}{P_m})$ with forcing a red path $P_k$ or $\left\lceil\frac{m}{n-v(H)}\right\rceil$ disjoint blue paths on the $n$ vertices within $\left\lceil\frac{m}{n-v(H)}\right\rceil \rr(\textcolor{red}{H},\textcolor{blue}{P_n})$ rounds. Then he applies the previous lemma $\left\lceil\frac{m}{n-k}\right\rceil-1$ times, each time merging two blue paths into one (or creating a red $H$). By each merge we lose not more than $v(H)$ vertices from the blue paths and it costs at most $e(H)$ rounds, so we end up with blue path with at least $\left\lceil\frac{m}{n-v(H)}\right\rceil (n-v(H))+v(H)\ge m$ vertices obtained within at most $\left\lceil\frac{m}{n-v(H)}\right\rceil(\rr(\textcolor{red}{H},\textcolor{blue}{P_n})+e(H))$ rounds in total.
\end{proof}
\begin{theorem}
    For any bipartite graph $H$, the limit
    $$\lim_{n\to\infty} \frac{\rr(\textcolor{red}{H},\textcolor{blue}{P_n)}}{n}$$
    exists (and it is finite).
\end{theorem}

\begin{proof}
   Let $$a=\liminf_{n\to\infty} \frac{\rr(\textcolor{red}{H},\textcolor{blue}{P_n})}{n}=\lim_{l\to\infty} \frac{\rr(\textcolor{red}{H},\textcolor{blue}{P_{n_l}})}{n_l}.$$
Note that from the previous corollary we have (by substituting $n=v(H)+1$) that
$$\frac{\rr(\textcolor{red}{H},\textcolor{blue}{P_m})}{m}\le  \rr(\textcolor{red}{H},\textcolor{blue}{P_{v(H)+1}})+e(H),$$
so $a<\infty$.
   
   From the definition of the subsequence $n_l$ we have $\rr(H,P_{n_l})=a{n_l}+o(n_l)$. Consider a subsequence $(m_l)$ of natural numbers such that $m_l=\omega(n_l)$ (i.e. $n_l=o(m_l)$). Let us recall that $H$ is fixed, so $v(H)$ and $e(H)$ are constant. From the previous corollary we have
   $$\rr(\textcolor{red}{H},\textcolor{blue}{P_{m_l}})\le \left\lceil\frac{m_l}{n_l-v(H)}\right\rceil(\rr(\textcolor{red}{H},\textcolor{blue}{P_{n_l}})+e(H))=\left\lceil\frac{m_l}{n_l-v(H)}\right\rceil(an_l+o(n_l))$$
   $$= \frac{m_l}{n_l-v(H)}(a(n_l-v(H))+o(n_l-v(H)))+O(n_l)=am_l+o(m_l)+O(n_l)=am_l+o(m_l).$$
Suppose that $\limsup_{m\to\infty}\frac{\rr(\textcolor{red}{H},\textcolor{blue}{P_m})}{m}>a.$ Then there is a sequence $(q_k)$ such that $$\lim_{k\to\infty}\frac{\rr(\textcolor{red}{H},\textcolor{blue}{P_{q_k}})}{q_k}>a.$$
and a subsequence $(q_{k_l})=(m_l)$ such that $m_l=\omega(n_l)$, so we have a contradiction.

   Thus $a$ is asymptotically both lower limit and upper bound of $ \frac{\rr(\textcolor{red}{H},\textcolor{blue}{P_m})}{m}$, and hence 
$$\lim_{m\to\infty} \frac{\rr(\textcolor{red}{H},\textcolor{blue}{P_m})}{m}=a.$$
\end{proof}

%Since $\rr(\textcolor{red}{H},\textcolor{blue}{P_m})\ge e(H)+e(P_m)-1\ge m-1$ for any nonempty graph $H$, we have that $\rr(\textcolor{red}{H},\textcolor{blue}{P_m})=\Theta(m)$ for any bipartite nonempty graph $H$.
\section{Remarks and open problems}
\begin{center}
\begin{tabular}{|*{10}{c|}}\hline
\backslashbox{bound}{k} &\cellcolor{red!25} 2 &\cellcolor{red!25} 3 &\cellcolor{red!25} 4 &\cellcolor{red!25} 5 &\cellcolor{red!25} 6 &\cellcolor{red!25} 7 &\cellcolor{red!25} 8 &\cellcolor{red!25} $9^+$  \\\hline
\cellcolor{red!25}lower bound of $\lim_{n\to\infty}  \rr(P_k,P_n)/n$  &  \cellcolor{red!10} &\cellcolor{red!10} &\cellcolor{red!10}  &\cellcolor{red!10} 1.5 & \cellcolor{red!10}1.5 &\cellcolor{red!10} 1.6 & \cellcolor{red!10}1.(63) & \cellcolor{red!10} \\ \cline{1-1} \cline{5-8}
\cellcolor{red!25}upper bound of $\lim_{n\to\infty}  \rr(P_k,P_n)/n$  & \cellcolor{red!10}\multirow{-2}{*}{1} &\cellcolor{red!10}\multirow{-2}{*}{1.25} & \cellcolor{red!10}\multirow{-2}{*}{1.4}& \multicolumn{4}{c|}{\cellcolor{red!10}1.(6)} &\cellcolor{red!10}\multirow{-2}{*}{1.(6)} \\\hline
\end{tabular}
\end{center}

We initially expected that the sequence $(\lim_{n\to\infty}  \rr(\textcolor{red}{P_k},\textcolor{blue}{P_n})/n)_{k=1}^\infty$ is concave, i.e. that the speed of growth slows down, when $k$ increases. We cannot rule out this possibility, however we could not find a better strategy for Painter in $\RR(\textcolor{red}{P_6},\textcolor{blue}{P_n})$ than the one presented in \cite{lo}. From the other side, our method gave us optimal lower bounds of  $\lim_{n\to\infty}  \rr(\textcolor{red}{P_k},\textcolor{blue}{P_n})/n$ for any fixed $k$, for which value of this limit is known (for example, see Exercise \ref{k=3} for $k=3$). This is why we state the following conjecture.

\begin{conj}
    Let 
    $$a_k=\lim_{n\to\infty}  \rr(\textcolor{red}{P_k},\textcolor{blue}{P_n})/n.$$
    Then $a_5=a_6=1.5, a_7=1.6,a_8=1.(63)$. 
\end{conj}

\bibliographystyle{amsplain}
 
\appendix
\section{The method of creating strategies and potential functions}
We were using linear optimization solver from the website \url{https://online-optimizer.appspot.com/}. We made the following assumptions:
\begin{itemize}
    \item If $i<j$ and a vertex can move from $X$ to $Z$ after adding a red edge, then $c_T(Z^j)-c_T(X^j)\le c_T(Z^i)-c_T(X^i)$.
    \item $c_T(X^{2+})-c_T(X^1)=c_T(X^{1})-c_T(X^0)$ for $X\in \{L_0,L_1,N_0,N_1,F_i\}$.
\end{itemize}
We fixed the ``jump'', i.e. the maximum change of the potential in one move and maximize the minimum of numbers in the $2^+$ row. We were using the following hill climbing method of optimizing the Painter's strategy:
\begin{enumerate}
    \item Start with some strategy, run the model.
    \item Find a constraint with non-zero dual value that corresponds to Painter's choice in a specific situation.
    \item Change the strategy in this place and rerun the model.
    \item If the objective function is worse, go back to the previous strategy.
    \item Go to the point 2.
\end{enumerate}
We put our final models in \url{https://github.com/urojony/PkPn}. %We do not recommend reading the code, unless you would like to create a similar model and do not know how to start.

\section{Known results}
In the table, we present known lower and upper bounds for $\rr(P_k,P_n)$.

\begin{table}[H]
\begin{center}
\begin{tabular}{|l|ll|l}
\hline
\cellcolor{red!25}$k$ & \multicolumn{2}{l|}{\cellcolor{red!25} $\rr(P_k,P_n)$}  \\ \hline
 & \multicolumn{1}{l|}{lower} & upper  \\ \hline
\cellcolor{red!25}3 & \multicolumn{2}{l|}{\begin{tabular}[c]{@{}l@{}}$\lceil 1.25(n-1)\rceil$ for $n\ge 3$ \\ \cite[Thm 1.5]{lo}\end{tabular}}  \\ \hline
\cellcolor{red!25}4 & \multicolumn{2}{l|}{\begin{tabular}[c]{@{}l@{}}$\lceil 1.4n-1\rceil$ for $n\ge 3$ \\ \cite[Thm 1.6]{lo} $\wedge$ (\cite[Thm 1.4]{Gosia} $\vee$ \cite{zhang}) \end{tabular}}  \\ \cline{1-4}
\cellcolor{red!25}7 & \multicolumn{1}{l|}{{\color[HTML]{009901} $1.6n-2$     (Cor \ref{pkpn})}} &  \\ \cline{1-2} \cline{4-4}
\cellcolor{red!25}8 & \multicolumn{1}{l|}{{\color[HTML]{009901} \begin{tabular}[c]{@{}l@{}}$1.(63)n-2$\\ (Cor \ref{pkpn})\end{tabular}}} &  \\ \cline{1-2} \cline{4-4}
\cellcolor{red!25}9 & \multicolumn{1}{l|}{{\color[HTML]{009901} \begin{tabular}[c]{@{}l@{}}$1.(6)m-2$\\ (Cor \ref{pkpn})\end{tabular}}} &  \\ \cline{1-2} \cline{4-4}
\cellcolor{red!25}$k$ & \multicolumn{1}{l|}{\begin{tabular}[c]{@{}l@{}}{\color{red}$1.(6)n+ k/8-4$ for $k\ge 9$}\\ (Cor \ref{pkpn})$\wedge$ Ex \ref{1/8})
\\ $1.(6)n+k/9-4$\\ \cite[Thm 1.2]{adva}\\ $1.5 n+0.5k-4$\\ 
\cite[Thm 1.4]{lo}
\end{tabular}
}
& \multirow{-7}{*}
{\begin{tabular}[c]{@{}l@{}}$1.(6)n+12k$ \\ \cite[Thm 1.2]{Gosia}\\ $2n+2k-7$ \\ \cite[Thm 2.3]{gryt}\end{tabular}}   \\ \hline
\end{tabular}

\end{center}
\end{table}

\section{Exercises}
\begin{exe}\label{k=3} Let a potential function $f$ be given by formula $f(G)=\sum_{v\in G} c(v).$
Find a generator $c:V(G)\to \mathbb{R}$ of a potential function $f$ and a Painter's strategy such that:
\begin{itemize}
    \item if $u\in O^0$, then $c(u)=0$,
    \item if $u\in V^1$, then $c(u)\ge 5$,
    \item if $u\in V^{2+}$, then $c(u)\ge 10$,
    \item if a host graph $G$ does not have a red $P_3$ nor any red cycle and Painter follows the strategy, then the potential of the host graph does not increase by more than 8 in every move, i.e. $f(G')-f(G)\le 8$.
\end{itemize}

Conclude that $\rr(\textcolor{blue}{\cC\cup \{P_3\}},\textcolor{red}{\cH(5/4,-5/8,x)})\ge x$ and, in particular, $\rr(\textcolor{red}{P_3},\textcolor{blue}{P_n})\ge 5n/4-5/4$.

~

\textbf{Hint.} Define the generator $c$ by filling out the following table:
\begin{center}
\begin{tabular}{|r||*{6}{c|}}\hline
\backslashbox{\textcolor{blue}{blue}}{\textcolor{red}{red}}
&\cellcolor{red!25}$O$&\cellcolor{red!25}$I$ \\\hline\hline
\cellcolor{blue!25}0&0&\cellcolor{gray!25}\\\hline
\cellcolor{blue!25}1&\cellcolor{gray!25}&\cellcolor{gray!25}\\\hline
\cellcolor{blue!25}$2^+$&\cellcolor{gray!25}&\cellcolor{gray!25}\\\hline
\end{tabular}
    
\end{center}

\end{exe}
\begin{exe}\label{k=5}
Let a potential function $f$ be given by formula $f(G)=\sum_{v\in G} c(v).$
Find a generator $c:V(G)\to \mathbb{R}$ of a potential function $f$ and a Painter's strategy such that:
\begin{itemize}
    \item if $u\in O^0$, then $c(u)=0$,
    \item if $u\in V^1$, then $c(u)\ge 3$,
    \item if $u\in V^{2+}$, then $c(u)\ge 6$,
    \item if a host graph $G$ does not have a red $P_4$ nor any red cycle and Painter follows the strategy, then the potential of the host graph does not increase by more than 4 in every move.
\end{itemize}

Conclude that $\rr(\textcolor{red}{\cC\cup \{P_5\}},\textcolor{blue}{\cH(3/2,-3/4,x)})\ge x$ and, in particular, $\rr(\textcolor{red}{P_5},\textcolor{blue}{P_n})\ge 3n/2-3/2$.

~

\textbf{Hint.} Use the following table. Try to guess the meaning of $F_0$ and $F_1$ in this game and define the generator $c$ by filling out the table.
\begin{center}
\begin{tabular}{|r||*{6}{c|}}\hline
\backslashbox{\textcolor{blue}{blue}}{\textcolor{red}{red}}
&\cellcolor{red!25}$O$&\cellcolor{red!25}$F_0$&\cellcolor{red!25}$F_1$ \\\hline\hline
\cellcolor{blue!25}0&0&\cellcolor{gray!25}&\cellcolor{gray!25}\\\hline
\cellcolor{blue!25}1&\cellcolor{gray!25}&\cellcolor{gray!25}&\cellcolor{gray!25}\\\hline
\cellcolor{blue!25}$2^+$&\cellcolor{gray!25}&\cellcolor{gray!25}&\cellcolor{gray!25}\\\hline
\end{tabular}
    
\end{center}

\end{exe}

\begin{exe}\label{1/8}
~
    \begin{enumerate}
        \item Let $G,H$ be graphs. Suppose that for any subgraph $H'\subset H$ with $|E(H\setminus H')|=n$, $H'$ contains a isomorphic copy of $G$. Show that for any host graph $J$ and a graph $K$ if $\rr_J(G,K)>0$, then 
    $$\rr_J(\textcolor{red}{H},\textcolor{blue}{K})\ge \rr_J(\textcolor{red}{G},\textcolor{blue}{K})+n.$$
    \item Conclude that for any nonempty graph $G$
    $$\rr(\textcolor{red}{G},\textcolor{blue}{C_n})\ge \rr(\textcolor{red}{G},\textcolor{blue}{P_n})+1,$$
    $$\rr(\textcolor{red}{P_n},\textcolor{blue}{G})\ge \rr(\textcolor{red}{P_k},\textcolor{blue}{G})+\lfloor (n-1)/(k-1)\rfloor -1.$$
    \end{enumerate}
\end{exe}
\newpage
\section{Answers}
\textbf{C.1}.
Painter's strategy is to color the edge blue if and only if at least one of its ends is in $I$. The potential generator $c$ is defined by the following table.
\begin{center}
\begin{tabular}{|r||*{6}{c|}}\hline
\backslashbox{\textcolor{blue}{blue}}{\textcolor{red}{red}}
&\cellcolor{red!25}$O$&\cellcolor{red!25}$I$ \\\hline\hline
\cellcolor{blue!25}0&0&4\\\hline
\cellcolor{blue!25}1&5&7\\\hline
\cellcolor{blue!25}$2^+$&10&10\\\hline
\end{tabular}
\end{center}

~
\newline
\textbf{C.2}. Let us remind that $O$ is the set of isolated vertices in red host graph. For every connected component with at least one edge in the red host graph, let the capital edge of this component be the first edge that was chosen by Builder. Let $F_0$ be the set of endpoints of the capitals and $F_1$ -- all vertices that are not in $F_0$ nor $O$. Painter's strategy is to color the edge red if and only if one of the vertices is in $O$ and the other in $O\cup F_0$. The potential generator $c$ is defined by the following table.
\begin{center}
\begin{tabular}{|r||*{6}{c|}}\hline
\backslashbox{\textcolor{blue}{blue}}{\textcolor{red}{red}}
&\cellcolor{red!25}$O$&\cellcolor{red!25}$F_0$&\cellcolor{red!25}$F_1$ \\\hline\hline
\cellcolor{blue!25}0&0&2&4\\\hline
\cellcolor{blue!25}1&3&4&5\\\hline
\cellcolor{blue!25}$2^+$&6&6&6\\\hline
\end{tabular}
    
\end{center}

~\newline
\textbf{C.3.1}. We use the strategy for Painter in the game $\RR_J(\textcolor{red}{G},\textcolor{blue}{K})$ to avoid both red $G$ and blue $K$ for $\rr_J(\textcolor{red}{G},\textcolor{blue}{K})-1$ moves. After that, Painter colors the next $n$ edges red. By the assumption the host graph does not have a red $H$. It means that 
$$\rr_J(\textcolor{red}{H},\textcolor{blue}{K})> \rr_J(\textcolor{red}{G},\textcolor{blue}{K})+n-1.$$
\textbf{C.3.2}. Removing any edge from the $C_n$ gives us $P_n$. A $P_n$ path contains $\lfloor (n-1)/(k-1)\rfloor$ edge-disjoint $P_k$ paths, therefore after removing any $\lfloor (n-1)/(k-1)\rfloor-1$ edges it still contains $P_k$ as a subgraph.

\end{document}